\documentclass[12pt,oneside]{amsart}

 \usepackage{amssymb} 
\theoremstyle{plain}  
      \newtheorem{theorem}{Theorem}[section]
      \newtheorem{lemma}[theorem]{Lemma}
      \newtheorem{corollary}[theorem]{Corollary}
      \newtheorem{proposition}[theorem]{Proposition}
      \theoremstyle{definition}
      
       \newtheorem{example}{Example}[section] 
      \newtheorem{remark}[theorem]{Remark}
 
      \usepackage[usenames,dvipsnames]{color}
   \usepackage[pdftex]{graphicx}
 \usepackage{colortbl} 
      
  \usepackage{wrapfig}
 \usepackage{caption}{\upshape}

\begin{document}
\author{N. Ghroda}

  \address{Department of Mathematics\\University
  of York\\Heslington\\York YO10 5DD\\UK}
\email{ng521@york.ac.uk}

    \title[Bicyclic semigroups of left I-quotients]{Bicyclic semigroups of left I-quotients}

   \begin{abstract}In this article we study left I-orders in the bicyclic monoid $\mathcal{B}$. We give necessary and sufficient conditions for a subsemigroup of $\mathcal{B}$ to be a left I-oreder in $\mathcal{B}$. We then prove that any left I-order in  $\mathcal{B}$  is straight.
   \end{abstract}
   

   \date{\today}

   \maketitle
 
\section{Introduction}
 The first published description of the bicyclic semigroup was given by Evgenii Lyapin in 1953 \cite{Lyp}. A description of the subsemigroups of the bicyclic monoid was given in 2005 \cite{ruskuc}. In this article, we use this description to study  left I-orders in the bicyclic monoid. \par
\bigskip
Many definitions of semigroups of quotients have been proposed and studied. The first, that was specifically tailored to the structure of semigroups was introduced  by Fountain and Petrich in \cite{pjhon}, but was restricted to completely 0-simple semigroups of left quotients. This definition has been extended to the class of all semigroups \cite{bisGould}. The idea is that a subsemigroup  $S$  of a semigroup $Q$  is a \emph{left order} in  $Q$   or  $Q$  is a \emph{semigroup of left quotients} of $S$ if every element of  $Q$  can be written as   $a^{\sharp}b$  where  $a , b \in S$  and  $a^{\sharp}$  is the inverse of   $a$  in a subgroup of  $Q$ and if, in addition, every \emph{square-cancellable} element  (an element  $a$ of a semigroup  $S$ is square-cancellable if  $a\, \mathcal{H}^{*}\,a^{2}$) lies in a subgroup of  $Q$. \emph{Semigroups of right quotients} and \emph{right orders} are defined dually. If  $S$ is both a left order and a right order in a semigroup  $Q$, then  $S$  is an \emph{order} in  $Q$ and $Q$ is a semigroup of \emph{quotients} of $S$. This definition and its dual were used in \cite{bisGould} to characterize semigroups which have  bisimple inverse $\omega$-semigroups of left quotients.\par
\bigskip
On the other hand, Clifford \cite{clifford} showed that from any right cancellative monoid  $S$ with (LC) we can construct a bisimple inverse  monoid $Q$ such that  $Q=S^{-1}S$; that is, every element  $q$  in $Q$ can be written as  $a^{-1}b$   where  $a ,b \in S$ and $a^{-1}$ is the inverse of $a$ in $Q$ in the sense of inverse semigroup theory. By saying that a semigroup  $S$ has the (LC) \emph{condition} we mean that for any  $a,b\in S$ there is an element  $c\in S$ such that  $Sa\cap Sb=Sc$.  The author and Gould in \cite{GG} have  extended Clifford's work to a left ample semigroup with (LC) where they  introduced the following definition of left I-orders in inverse semigroups:\par
 \bigskip
 Let  $Q$ be an inverse semigroup. A  subsemigroup  $S$  of  $Q$  is a \emph{left I-order} in  $Q$ or $Q$ is a semigroup of \emph{left I-quotients} of $S$, if every element in  $Q$  can be written as  $a^{-1}b$    where  $a ,b \in S$. The notions of \emph{right I-order} and  \emph{semigroup of right I-quotients} are defined dually. If  $S$ is both a left I-order and a right I-order in  $Q$, we say that  $S$  is an \emph{I-order} in  $Q$ and $Q$ is a semigroup of \emph{I-quotients} of $S$. It is clear that, if $S$ a left order in an inverse semigroup $Q$, then it is certainly a left I-order in $Q$; however, the converse is not true   (see for example \cite{GG} Example 2.2). \par
\bigskip 
A left I-order  in an inverse semigroup $Q$ is  \emph{straight left I-order} if every element in  $Q$ can be written as  $a^{-1}b$ where  $a,b \in S$ and  $a\,\mathcal{R}\,b$ in  $Q$; we also say that $Q$ is a \emph{straight left I-quotients} of $S$. If  $S$ is  straight in $Q$,
we have the advantage of controlling the product in $Q$.  \par
\bigskip
In \cite{NG} the author has given the necessary and sufficient conditions for a semigroup $S$ to have a  bisimple inverse $\omega$-semigroup left I-quotients, modulo left I-order in the bicyclic semigroup $\mathcal{B}$, which is the most straightforward example of the bisimple inverse $\omega$-semigroup. In fact, it is a semigroups with many remarkable properties. Left I-orders in the bicyclic semigroup are interesting in their own right.  By describtions left I-order in $\mathcal{B}$, we obtain:
\begin{theorem}\label{main} Let $S$ be a subsemigroup of $\mathcal{B}$. If $S$ is a left I-order in $\mathcal{B}$, then it is straight. \end{theorem}

In the preliminaries after introducing the necessary notation, we give some previous results giving the description of subsemigroups of $\mathcal{B}$.\par
\bigskip
We use the classification of  subsemigroups of $\mathcal{B}$ in \cite{ruskuc} to investigate which of them are left I-orders in  $\mathcal{B}$. Subsemigroups of $\mathcal{B}$ fall into three classes upper, lower and two-sided. In Sections 3, 4 and 5 we give the necessary and  sufficient conditions for  upper, lower and two-sided subsemigroups of $\mathcal{B}$ to be left I-orders in $\mathcal{B}$, respectively. In each case, such left I-orders are straight and this proves Theorem ~\ref{main}.

\newpage

\section{Preliminaries}\label{prelim}
Throughout this article we shall follow the terminology and notation of \cite{clifford}. The symbol $\mathbb{N}$ will denote the set consisting of the natural numbers and $\mathbb{N}^0=\mathbb{N}\cup \{0\}$. Let $ \mathcal{R} ,  \mathcal{L} ,  \mathcal{H}$ and  $\mathcal{D}=  \mathcal{R} \circ \mathcal{L}=\mathcal{L} \circ \mathcal{R}$\ be the usual Green's relations. A semigroup $S$ is called \emph{simple} if $S$ does not contain proper two-sided ideals and \emph{bisimple} if it consists of a single  $\mathcal{D}$-class. \par
 \medskip 
The bicyclic semigroup  $\mathcal{B}(a,b)$ is defined by the monoid generated by two elements $a$ and  $b$ subject only to the condition that $ba=1$. It follows that the elements can all be written in the standard form $a^ib^j$ where $i,j \geq 0$. We can write out the elements of  $\mathcal{B}$ in array. 
\begin{center}\[\begin{array}{c|ccccc}
1& b & b^{2} & b^{3} & b^{4} & \ldots\\ \hline
a & ab & ab^{2} & ab^{3} & ab^{4} & \ldots \\ 
a^{2} & a^{2}b & a^{2}b^{2} & a^{2}b^{3}  & a^{2}b^{4} & \ldots\\
a^{3}& a^{3}b &a^{3}b^{2} & a^{3}b^{3} & a^{3}b^{4} & \ldots \\
a^{4} & a^{4}b& a^{4}b^{2} & a^{4}b^{3} & a^{4}b^{4}  & \ldots \\
\vdots & \vdots & \vdots & \vdots & \vdots & \ddots\end{array}\]\end{center}
The  multiplication on  $\mathcal{B}$ is defined as follows:
\[a^{k}b^{l}a^{m}b^{n} = \begin{cases}
a^{k+m-l}b^{n} & l\leq m, \\
a^{k}b^{l-m+n} & l> m.  
\end{cases}\] 
We can put the two cases together as follows: \[a^kb^la^mb^n =a^{k-l+t}b^{n-m+t}\ \mbox{where}\ t=\mbox{max}\{l,m\}.\]
The monoid $\mathcal{B}$ is thus isomorphic to the monoid $\mathbb{N}^0 \times \mathbb{N}^0$ with multiplication 
\[(k,l)(m,n)=(k-l+t,n-m+t)\ \mbox{where}\ t=\mbox{max}\{l,m\}.\] 
It is easy to see that $\mathcal{B}$ is an inverse semigroup: the element $a^ib^j$ has inverse $a^jb^i$. The idempotents of  $\mathcal{B}$ are of the form \[e_n=a^nb^n \;  (n=0,1,2,...) \; \mbox{which satisfy} \; 1=e_0\geq e_1 \geq e_2 \geq ....\]

Green's relations $\mathcal{L} , \mathcal{R}$  and  $\mathcal{H}$  are given by
\[a^{i}b^{j}\,\mathcal{L}\,a^{k}b^{l} \; \mbox{if and only if}\; j = l,\] 
\[a^{i}b^{j}\,\mathcal{R}\, a^{k}b^{l} \; \mbox{if and only if} \; i = k,\]
and
\[a^{i}b^{j}\,\mathcal{H}\,a^{k}b^{l} \; \mbox{if and only if}\; i = k\; \mbox{and}\; j=l.\]  

In the array, the rows are the $\mathcal{R}$-classes of $\mathcal{B}$, the columns are the $\mathcal{L}$-classes and
the $\mathcal{H}$-classes are points. There is only one $\mathcal{D}$-class; that is, $\mathcal{B}$ is a bisimple monoid (hence simple).\par
\bigskip
     
Following \cite{ruskuc}, we start by introducing some basic subsets of  $\mathcal{B}$,
\[\begin{array}{rcl}D &= &\{a^{i}b^{i} : i \geq 0\} \:............................................ \mbox{the \emph{diagonal}}.\\
L^{p} &=& \{a^{i}b^{j} : 0 \leq j \leq p, i \geq 0\}\; \mbox{for} \; p \geq 0\: ............. \mbox{the \emph{left strip} (determined by $p$)}. \end{array}\]  For \ $0 \leq q \leq p$ we define the \emph{triangle}
\[T_{q,p} = \{a^{i}b^{j} : q \leq i \leq j < p\}.\]
 For  $i,m \geq 0$ and  $d > 0$ we define the rows
\[\Lambda_{i} = \{a^ib^j : j \geq 0\},\ \Lambda_{i,m,d} = \{a^{i}b^{j} : d|j- i, j \geq m\}\]
and in general for  $I\subseteq \{0, \ldots ,m -1\}$,
 \[\Lambda_{I,m,d} =\bigcup_{i \in I}\Lambda_{i,m,d}=\{a^{i}b^{j}:i\in I, d|j-i , j\geq m\}.\]
For  $p \geq 0, d > 0, r \in [d] = \{0,\ldots, d-1\}$ and  $P \subseteq[d]$ we define the \emph{squares}
\[\Sigma_{p} = \{a^{i}b^{j}:i,j\geq p\}, \ \Sigma_{p,d,r} = \{a^{p+r+ud}b^{p+r+vd} : u, v\geq 0\}, \]
\[\Sigma_{p,d,P}=\bigcup_{r\in P}\Sigma_{p,d,r}=\{a^{p+r+ud}b^{p+r+vd} : r \in P, u, v \geq 0\}.\]
It is worth pointing out that in \cite{ruskuc} it was shown that a subsemigroup of $\mathcal{B}$ is inverse if and only if it has the form $F_D \cup \Sigma_{p,d,P}$ where $F_D$ is a finite subset of the diagonal (which may be empty).
The function  $\widehat{} :\mathcal{B}\longrightarrow \mathcal{B}$ defined by  $a^{i}b^{j} \rightarrow \widehat{a^{i}b^{j}}= a^{j}b^{i}$ is an anti-isomorphism. Geometrically it is the reflection with respect to the main diagonal.
\begin{figure}[htp]
\centering 
\includegraphics[height=2.2in]{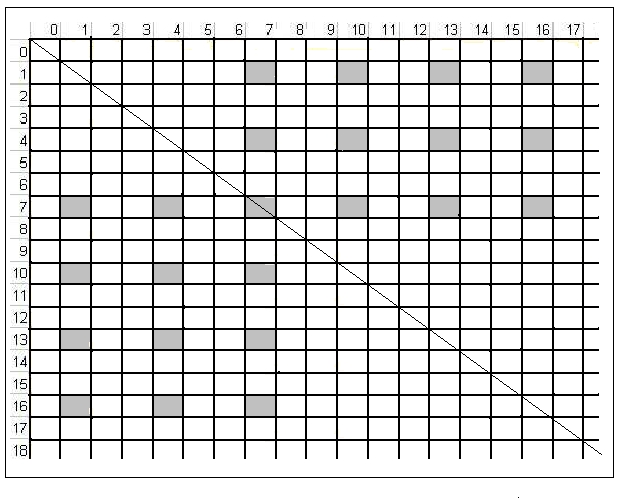}
\includegraphics[height=2.2in]{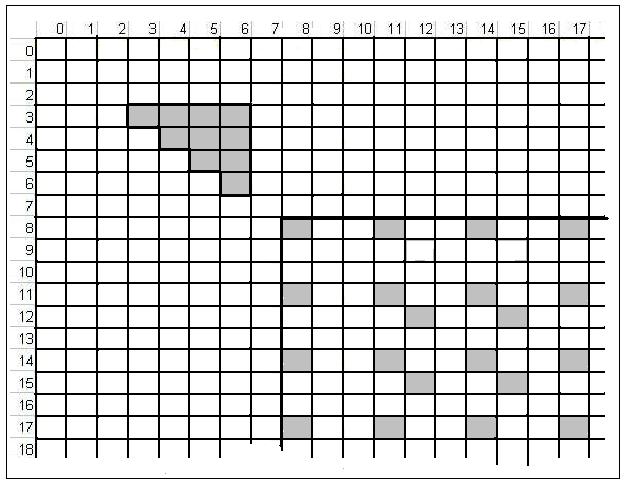}
\caption{The subsets $\Lambda_{\{1,4,7\},7,3}, \widehat{\Lambda}_{\{1,4,6\},7,3},\Sigma_{8,3,\{0,4\}}$ and $T_{3,7}$}
\end{figure} 
\begin{proposition}\label{subbicyclic}\cite{ruskuc}
Let $S$ be a subsemigroup of the bicyclic monoid. Then one of the following conditions holds:\\
1. $S$ is a subset of the diagonal,  $S \subseteq D$.\\
2. $S$ is a union of a subset of a triangle, a subset of the diagonal above the triangle, a square below the triangle and some lines belonging to a strip determined by the square and the triangle, or the reflection of this union with respect to the diagonal. Formally there exist  $q, p \in \mathbb{N}^0$ with  $q \leq p, d \in \mathbb{N},   I \subseteq \{q, . . . , p-1\}$ with  $q \in I, P \subseteq \{0, . . . , d - 1\}$ with  $0 \in P, \ F_{D} \subseteq D \cap L_{q}, \ F \subseteq T_{q,p}$  such that  $S$ is of one of the following forms:
\begin{center}
(i) \ $S = F_{D} \cup F \cup \Lambda_{I,p,d}\cup \Sigma_{p,d,P}$;\\
(ii)\ $S = F_{D} \cup \widehat{F} \cup \widehat{\Lambda}_{I,p,d}\cup \Sigma_{p,d,P}$.\\
\end{center}
3. There exist  $d \in \mathbb{N},\ I \subseteq \mathbb{N}^0,\ F_{D} \subseteq D \cap L_{min(I)}$ and sets  $S_{i} \subseteq \Lambda_{i,i,d}\ (i \in I)$ such that  $S$ is of one of the following forms:
\begin{center}
\; (i)\ $S = F_{D}\cup\bigcup_{i\in I} S_{i}$; \\ 
(ii) $S = F_{D}\cup\bigcup_{i\in I} \widehat{S_{i}}$;
\end{center}
where each  $S_{i}$ has the form 
\[S_{i} = F_{i} \cup \Lambda_{i,m_i,d}\]
for some  $m_{i} \in  \mathbb{N}^0$ and some finite set  $F_{i}$, and
\[I = I_{0} \cup \{r + ud : r \in R, u \in  \mathbb{N}^0, r + ud \geq N\}\]
for some (possibly empty)  $R \subseteq \{0, . . . , d - 1\}$, some  $N \in \mathbb{N}^0$ and some finite set
\ $I_{0} \subseteq \{0, . . . ,N - 1\}.$\\
\\
We call diagonal subsemigroups those defined by 1., two-sided subsemigroups those defined by 2., upper subsemigroups  those defined by 3.(i) and lower subsemigroups those defined by 3.(ii). \end{proposition}

We begin with the following example which plays  a significant role in studying left I-orders in   
 $\mathcal{B}$.

\begin{example}\label{exRclass} 
 Let $R_1=\{a^0b^j: j\geq 0\}$ be the $\mathcal{R}$-class of the identity element 1 of $\mathcal{B}$ and $q=a^mb^n\in \mathcal{B}$. Then \[q=a^mb^n=(a^0b^m)^{-1}(a^0b^n),\] so that  $R_1$ is a straight left I-order in $\mathcal{B}$. In fact, it is a special case of Clifford's result,  mentioned in the Itroduction. 
\end{example}
\begin{remark}\label{rems}
Any subsemigroup of $\mathcal{B}$ that contains $R_1$ is a straight left I-order in $\mathcal{B}$.
\end{remark}

\begin{lemma}\label{identity}
Let $S$ be a left I-order in  $\mathcal{B}$. Then for any $\mathcal{L}$-class $L_k$ of $\mathcal{B}$, $S\cap L_k\neq \emptyset$. \end{lemma}
\begin{proof}
Let $k\in \mathbb{N}^0$. Then  

\[\begin{array}{rcl}a^kb^k&=&(a^ib^j)^{-1}(a^mb^n)\\ &=&a^jb^ia^mb^n\\ &=&
a^{j-i+t}b^{n-m+t}\end{array}\]
where $t=$max$\{i,m\}$, for some $a^ib^j,a^mb^n\in S$. Hence $k=j-i+t=n-m+t$, so that either $k=j$ or $k=n$. Thus $S\cap L_k\neq \emptyset$.\end{proof}

We conclude this section by the following lemma which plays a significant role in the next sections.

\begin{lemma}\label{dimpact} Let $S$ be a left I-order in $\mathcal{B}$ and let $d\in \mathbb{N}$. If for all $a^ib^j \in S$ we have $d|i-j$, then $d=1$.\end{lemma}
\begin{proof}
 Let  $a^kb^l\in \mathcal{B}$. Then there exist $a^ib^j,a^mb^n \in S$ with
 
\[\begin{array}{rcl}  a^kb^l&=&(a^ib^j)^{-1}(a^mb^n)\\ &=&a^jb^ia^mb^n\\ &=& a^{j-i+t}b^{n-m+t} \end{array}\]
where $t=$max$\{i,m\}$. Now \[\begin{array}{rcl} k-l&=&(j-i+t)-(n-m+t)\\ &=&(j-i)+(m-n)\equiv 0 (\mbox{mod} \ d).\end{array}\] It follows that $d=1$.\end{proof}

\section{Upper subsemigroups}\label{leftiupper}

In this section we give  necessary and  sufficient conditions for an upper subsemigroup $S$ of $\mathcal{B}$ to be a  left I-order in $\mathcal{B}$. The upper subsemigroups of $\mathcal{B}$ are those having all elements above the diagonal; that is, all elements satisfy:  $a^{i}b^{j}, j \geq i$. Throughout this section $S$ is an upper subsemigroup of $\mathcal{B}$ having the form  (3).($i$) in Proposition~\ref{subbicyclic}. We have already met one of them, which is the $\mathcal{R}$-class of the identity. By Lemma ~\ref{identity}, we deduce that any left I-order upper subsemigroup is a monoid. \par
\bigskip 

The next example is of a subsemigroup bigger than $\mathcal{R}$-class of the identity. In fact, it is the largest upper subsemigroup of $\mathcal{B}$.
\begin{example}
The upper subsemigroup  $\mathcal{B}^+=\{a^{i}b^{j}: j \geq i\}$ of $\mathcal{B}$ is a straight left I-order in  $\mathcal{B}$, by Remark~\ref{rems} as $R_1\subseteq \mathcal{B}^+$. In fact, we can write any element $a^ib^j$ of $\mathcal{B}$ as follows \[  a^ib^j=(a^ib^{i+j})(a^jb^{i+j})^{-1}\] where $a^ib^{i+j},  a^jb^{i+j} \in \mathcal{B}^{+}$, that is, $\mathcal{B}^{+}$ is a right I-order in $\mathcal{B}$. Hence $\mathcal{B}^+$ is an I-order in $\mathcal{B}$. It is worth pointing out that $\mathcal{B}^{+}$   is a full subsemigroup of $\mathcal{B}$ in the sense that $E(\mathcal{B})=E(\mathcal{B}^+)$. 
\end{example}

\begin{remark}
Let $S$ be an upper subsemigroup of $\mathcal{B}$. If $i\notin I$, then $S$ does not contain any element of form $a^ib^j$ for all $j> i$, and only contains $a^ib^i$ if $a^ib^i\in F_D$.
\end{remark}  

\begin{lemma}\label{zeroinI}
Let $S$ be an upper subsemigroup of $\mathcal{B}$. If $S$ is a left I-order in $\mathcal{B}$, then $d=1$ and $0\in I$.\end{lemma}
\begin{proof}
Since $S$ an upper subsemigroup, it follows that for all $a^ib^j \in S$ we have that $d|j-i$ for some $d\in \mathbb{N}$. By Lemma~\ref{dimpact}, it is clear that $d=1$. It is therefore remains to show that $0\in I$. By Lemma~\ref{identity}, we have that $1\in S$. Let $a^0b^h \in \mathcal{B} $ for some $h\in \mathbb{N}$. Hence 
\[\begin{array}{rcl}
a^0b^h&=&(a^ib^j)^{-1}(a^mb^n)\\ &=&a^{j-i+t}b^{n-m+t} \end{array}\]
where $t=$max$\{i,m\}$, for some $a^ib^j,a^mb^n\in S$. For $0=j-i+t$ we must have that $t=i$ and $j=0$. As $i\leq j$, it follows that $i=j=0$ and so $a^0b^h=a^mb^n\in S$.\end{proof}

The following corollary is obvious.
\begin{corollary}\label{RinI}
Let $S$ be an upper subsemigroup of $\mathcal{B}$. If $S$ is a left I-order in $\mathcal{B}$, then $R_1 \subseteq S$.\end{corollary}
 
By Lemma~\ref{identity},  $S\cap L_1\neq \emptyset$. As $F_D=D\cap L_{min(I)}$, the following corollary is clear.
\begin{corollary}\label{fd1}Let $S$ be an upper subsemigroup of $\mathcal{B}$. If $S$ is a left I-order in $\mathcal{B}$, then $F_D=\{1\}$ or $F_D= \emptyset$.\end{corollary}

We now come to the main result of this section which is the first result in this article.

\begin{proposition}\label{uppercondition}
For an upper subsemigroup $S$ of $\mathcal{B}$,  the following are equivalent: 

$(i)$  $S$ is a left I-order in $\mathcal{B}$; 

$(ii)$  $R_1 \subseteq S$.

Moreover, writing $S$ as $S=F_{D}\cup\bigcup_{i\in I} S_{i}$, we have $R_1 \subseteq S$ if and only if $0\in I, \ d=1$ and $F_D\cup F_0=\{1,...,a^0b^{m_0-1}\}$.\end{proposition}

\begin{proof}
The equivalence of $(i)$ and $(ii)$ follows from Example~\ref{exRclass} and Corollary~\ref{RinI}. The remaining statement follows from inspection of the description of $S$ as in $3(i)$ of Proposition~\ref{subbicyclic}.\end{proof}

\begin{corollary}\label{sublowbicyclic}  
Let  $S$ be an upper subsemigroup of $\mathcal{B}$. If  $S$ is a left I-order in $\mathcal{B}$, then  it is straight.\end{corollary}

\section{Lower subsemigroups}\label{leftilower}

In this section we give  necessary and  sufficient conditions for the lower subsemigroups of $\mathcal{B}$ to be  left I-orders in $\mathcal{B}$. Throughout this section $S$ is a lower subsemigroup of $\mathcal{B}$ having the form  (3).($ii$) in Proposition~\ref{subbicyclic}. We begin with:

\begin{example}The lower subsemigroup $T=\{a^i b^j : i\geq j ,  i \geq m\}$ of $\mathcal{B}$ is a straight left I-order in  $\mathcal{B}$. Since for any element $q=a^{k}b^{h}$ in  $\mathcal{B}$ we have \[q=a^{k}b^{h}=a^{k}b^{k+h+m}a^{k+h+m}b^{h}=(a^{k+h+m}b^{k})^{-1}(a^{k+h+m}b^{h})\] and it is clear that  $a^{k+h+m}b^{k}$ and $a^{k+h+m}b^{h}$ are in  $T$. \end{example}

\begin{remark}\label{R1}
Let $S$ be a lower subsemigroup of $\mathcal{B}$. If $j\notin I$, then $S$ contains no element $a^ib^j$ with $i>j$.\end{remark}

\begin{lemma}\label{zerolowecase}
Let $S$ be a lower subsemigroup of $\mathcal{B}$. If $S$ is a left I-order in $\mathcal{B}$, then $d=1$ and  $0\in I$.
\end{lemma}

\begin{proof}
Since $S$ a lower subsemigroup, it follows that for all $a^ib^j \in S$ we have that $d|j-i$ for some $d\in \mathbb{N}$. By Lemma~\ref{dimpact}, it is clear that $d=1$.  Let $a^hb^0\in \mathcal{B}$ where $h\in \mathbb{N}$. Then \[a^hb^0=(a^ib^j)^{-1}(a^mb^n)=a^{j-i+t}b^{n-m+t}\]  where $t=\max\{m,i\}$, so that $0=n-m+t$. Hence we deduce that $n=0$ and $t=m$. We also have that $h=j-i+m$ so that $m=h+(i-j)\geq h$ so that $a^mb^0\in S$. Hence $0\in I$.  \end{proof}

Since $F_D\subseteq D\cap L_{min(I)}$, the following corollary is clear.
\begin{corollary}\label{flo}Let $S$ be a lower subsemigroup of $\mathcal{B}$. If $S$ is a left I-order in $\mathcal{B}$, then $F_D=\{1\}$ or $F_D=\emptyset$.\end{corollary}

Suppose that a lower subsemigroup $S$ is a left I-order in $\mathcal{B}$. From Lemma \ref{zerolowecase}, we have that 
$d=1$ and  $0\in I$. We claim that $I=\mathbb{N}^0$. By Corollary ~\ref{flo}, $F_D=\{1\}$ or $F_D=\emptyset$, so that as $S$ intersects every $\mathcal{L}$-class of $\mathcal{B}$, by Lemma~\ref{identity}, we have that $I=\mathbb{N}^0$. We have one half of the following proposition.
\begin{proposition}\label{loweriorder}
A lower subsemigroup $S$ is a left I-order in $\mathcal{B}$ if and only if $d=1$ and $I=\mathbb{N}^0$.\end{proposition}
\begin{proof}
Suppose that $d=1$ and $I=\mathbb{N}^0$. Then \[\widehat{\Lambda}_{i,m_{i},1}=\{a^jb^i:j=t+i, j\geq m_i\}=\{a^{t+i}b^i:t+i\geq m_i\}.\]  For any $a^hb^k \in \mathcal{B}$ we have \[a^hb^k=(a^{h+k+t}b^h)^{-1}(a^{h+k+t}b^k)\] where $t=$max$\{m_h,m_k\}$ for $i\in \mathbb{N}^0$. It is clear that $a^{h+k+t}b^h,a^{h+k+t}b^k\in S$.
\end{proof}
The following corollary is clear from the proof of Proposition~\ref{loweriorder}
\begin{corollary}\label{stralow}
Let $S$ be a lower subsemigroup of $\mathcal{B}$. If $S$ is a left I-order in $\mathcal{B}$, then it is straight.\end{corollary}

\section{Two-sided subsemigroups}\label{leftitwo-sided}

In this section we give  necessary and  sufficient conditions for the two-sided subsemigroups of $\mathcal{B}$ to be left I-orders in $\mathcal{B}$. {The two-sided subsemigroups of $\mathcal{B}$ have the forms  (2).($i$) and (2).($ii$) in Proposition~\ref{subbicyclic}. Throughout this section we shall assume that a two-sided subsemigroup $S$ of $\mathcal{B}$ is proper, in the sense  $S\neq \mathcal{B}$.}\par
\bigskip  
We divide this section into two parts. We study the first form in the first part, and the second form in the second part. \par
\medskip
We begin with the two-sided subsemigroups which have the form (2).($i$) in Proposition~\ref{subbicyclic}.\par
\medskip
  Let $a^mb^n \in F \subseteq S=F_D \cup F  \cup \Lambda_{I,p,d}\cup \Sigma_{p,d,P}$. Then, $d|(n-m)$. For, as $0\in P$, $a^pb^{p+d}\in S$ and we have that  $a^pb^{p+d}a^mb^n=a^pb^{n-m+p+d} \in \Sigma_{p,d,P}$, so that $d|(m-n-d)$, that is, $m-n=(t+1)d$ for some $t\in \mathbb{N}^0$. Hence for any $a^ib^j \in S$ we have that $d|i-j$.  By Lemma ~\ref{dimpact}, the first part of the following lemma is clear.
\begin{lemma}\label{dequal1}
If a two-sided subsemigroup $S=F_{D}\cup F\cup \Lambda_{I,p,d} \cup \Sigma_{p,d,P}$ of $\mathcal{B}$ is a left I-order in  $\mathcal{B}$, then  $d=1$ and $q=0$. Consequently, $R_{1} \subseteq S$.\end{lemma}
\begin{proof}
Let $a^0b^h\in \mathcal{B}$ where $h\in \mathbb{N}$. Then, \[a^0b^h=(a^ib^j)^{-1}(a^mb^n)=a^{j-i+t}b^{n-m+t}\]  where $t=$max$\{m,i\}$, so that $0=j-i+t$. Hence we deduce that $j=0$. If $a^ib^j \in \Sigma_{p,d,P}$, then as $\Sigma_{p,d,P}$ is an inverse subsemigoup of $\mathcal{B}$ we have that $a^0b^h \in S$. In the case where $a^ib^j \notin \Sigma_{p,d,P}$ we must have that $a^ib^j\in F_{D}\cup F\cup \Lambda_{I,p,d}$. Hence $j\geq i$ so that $i=j=0$. It follows that $a^0b^h=a^mb^n\in S$. Hence $q=0$.  \end{proof}

Since $F_D\subseteq \{1\}$ we have that $F_D=\{1\}$ or $F_D=\emptyset$. In either case, $S\cap L_1=\{1\}$. Then the following corollaries are clear.
\begin{corollary}If a two-sided subsemigroup $S=F_{D}\cup F\cup \Lambda_{I,p,d} \cup \Sigma_{p,d,P}$ of $\mathcal{B}$ is a left I-order in  $\mathcal{B}$, then $F_D=\{1\}$ or $F_D=\emptyset$. \end{corollary}

\begin{corollary}\label{twosided}
A two-sided subsemigroup $S=F_{D}\cup F\cup \Lambda_{I,p,d} \cup \Sigma_{p,d,P}$ of $\mathcal{B}$ is a left I-order in  $\mathcal{B}$ iff  $R_{1} \subseteq S$.\end{corollary}

\begin{corollary}\label{twostraightr}
If $S=F_{D}\cup F\cup\Lambda_{I,p,d}\cup \Sigma_{p,d,P}$ is a left I-order in $\mathcal{B}$, then it is  straight.\end{corollary}

 Now, we start studying the second form which has the form (2).($ii$) in Proposition~\ref{subbicyclic}.\par
 \bigskip
Let $a^mb^n \in \widehat{F} \subseteq S=F_{D}\cup \widehat{F}\cup \widehat{\Lambda}_{I,p,d}\cup \Sigma_{p,d,P}$. Then, $d|n-m$. For, since $a^pb^{p+d} \in \Sigma_{p,d,P}$, it follows that $a^mb^na^pb^{p+d}=a^{m-n+p}b^{p+d} \in \Sigma_{p,d,P}$. Hence $d|(m-n-d)$, that is, $m-n-d=td$ for some $t \in \mathbb{N}^0$ and so $m-n=(t+1)d$. Hence for any $a^ib^j \in S$ we have that $d|i-j$. By Lemma~\ref{dimpact}, the first part of the following lemma is clear.

\begin{lemma}\label{dequal1l}
If a two-sided subsemigroup $S=F_{D}\cup \widehat{F}\cup \widehat{\Lambda}_{I,p,d}\cup \Sigma_{p,d,P}$  of $\mathcal{B}$ is a left I-order in  $\mathcal{B}$, then  $d=1$ and $q=0$. \end{lemma}

\begin{proof}
Suppose that $q\neq 0$, let $a^0b^k\in \mathcal{B}$ where $k\in \mathbb{N}$. Then \[a^0b^k=(a^ib^j)^{-1}(a^mb^n)=a^{j-i+t}b^{n-m+t}\]  where $t=\max\{m,i\}$, so that $0=j-i+t$. Hence we can deduce that $j=0$. If  $i=0$, then $a^0b^k=a^mb^n$ so that $a^0b^k\in S$ a contradiction and so $i>0$. Hence $a^ib^0\in S$, but $a^ib^0 \in \widehat{\Lambda}_{0,p,1}\cup \widehat{F}$ as $\widehat{F}\subset T_{0,p}$ a contradiction again. Therefore $q=0$ as required. \end{proof}

\begin{remark}\label{L1} 
In the case where $q=0$ it is easy to see that $a^pb^0\in S$. If $m\notin I$, then $a^ub^m\notin \widehat{F}$ for any $0\leq u<p$. For, if $a^ub^m \in \widehat{F}$, then $ a^pb^0a^ub^m =a^{p+u}b^m \in \widehat{\Lambda}_{I,p,d}$ a contradiction.
\end{remark}
   
\begin{proposition}\label{gcase}
The subsemigroup  $S=F_{D}\cup \widehat{F}\cup \widehat{\Lambda}_{I,p,d}\cup \Sigma_{p,d,P}$ of $\mathcal{B}$ is a left I-order in  $\mathcal{B}$ if and only if $d=1$ and $I=\{0,...,p-1\}$.\end{proposition}
\begin{proof}
($\Longrightarrow$) Suppose that $S$ is a left I-order in $\mathcal{B}$. Then any element $q=a^mb^n \in \mathcal{B}$ can be written as $(a^ib^j)^{-1}(a^kb^l)$ for some $a^ib^j,a^kb^l \in  S$. By Lemma~\ref{dequal1l}, $d=1$ and $0\in I$. It is remains to show that $I=\{0,...,p-1\}$.\par
\medskip
Let $0<m<p$. Then 
\[\begin{array}{rcl}
a^mb^m&=&(a^ib^j)^{-1}(a^kb^l)\\ &=& a^{j-i+t}b^{l-k+t}\end{array}\]  
where $t=$max$\{i,k\}$, for some $a^ib^j,a^kb^l\in S$. Then $m=j$ or $m=l$; so that $a^ub^m \in S$ for some $u$. If $m \notin I$, so $u< p$, then $a^pb^0a^ub^m=a^{p+u}b^m \in S$, in contradiction to Remark~\ref{L1}.

($\Longleftarrow$) Suppose that $d=1$ and $I=\{0,...,p-1\}$. Then for any $a^mb^n\in \mathcal{B}$ we have
\[a^mb^n=(a^{p+m+n}b^m)^{-1}(a^{p+m+n}b^n).\]It is clear that $a^{p+m+n}b^m,a^{p+m+n}b^n \in S$.
\end{proof} 

\begin{corollary}\label{twostraightl}
If $S=F_{D}\cup \widehat{F}\cup \widehat{\Lambda}_{I,p,d}\cup \Sigma_{p,d,P}$ is a left I-order in $\mathcal{B}$, then it is  straight.
\end{corollary}

From corollaries ~\ref{twostraightr} and ~\ref{twostraightl}, we have the main result in this section which is the third result in this article. 
\begin{corollary}
Let $S$ be a two-sided subsemigroup of $\mathcal{B}$. If $S$  is a left I-order in $\mathcal{B}$, then it is  straight.\end{corollary}

\end{document}